\documentclass[12pt]{article}
\RequirePackage{amsthm,amsmath,natbib,amssymb}
\let\cal=\mathcal
\let\Bbb=\mathbb
\def\bbb{{\Bbb B}}
\def\bbn{{\Bbb N}}
\def\bbr{{\Bbb R}}
\def\bbz{{\Bbb Z}}
\def\calp{{\cal P}}
\def\ep{\epsilon}
\def\uom{\Omega}
\def\udel{\Delta}
\def\lam{\lambda}
\def\tsim{\raise1pt\hbox{$\scriptscriptstyle\sim$}}
\def\var{\mathop{\rm var}}
\def\ess{\mathop{\rm ess}}
\def\essinf{\mathop{\rm ess\; inf}}
\def\esssup{\mathop{\rm ess\; sup}}
\def\remk{\noindent{\bf Remark.~~}}
\theoremstyle{definition}
\newtheorem{dfn}{Definition}[section]
\newtheorem{exam}[dfn]{Example}
\theoremstyle{plain}
\newtheorem{thm}[dfn]{Theorem}
\newtheorem{lem}[dfn]{Lemma}
\newtheorem{coro}[dfn]{Corollary}
\numberwithin{equation}{section}
\def\qqed{\qquad$\Box$}
\date{}

\begin{document}

\title{Conflations of Probability Distributions}

\author{Theodore P. Hill\\
School of Mathematics \\ Georgia Institute of Technology\\
Atlanta GA 30332-0160\\
{\tt hill@math.gatech.edu}
} 

\maketitle
\begin{abstract}
The {\it conflation} of a finite number of probability distributions 
$P_1,\dots, P_n$ is
a consolidation of those distributions into a single probability
distribution $Q=Q(P_1,\dots, P_n)$, where intuitively $Q$ is the conditional distribution
of independent random variables $X_1,\dots, X_n$ with distributions
$P_1,\dots, P_n$, respectively,
given that $X_1=\cdots =X_n$. Thus, in large classes of distributions the conflation
is the distribution determined by the normalized product of the
probability density or probability mass functions. $Q$ is shown to be the
unique probability distribution that minimizes the loss of Shannon
Information in consolidating the combined information from  $P_1,\dots, P_n$ into
a single distribution $Q$, and also to be the optimal consolidation
of the distributions with respect to two minimax likelihood-ratio
criteria.  When $P_1,\dots, P_n$ are Gaussian, $Q$ is Gaussian with mean the classical
weighted-mean-squares reciprocal of variances. A version of the
classical convolution theorem holds for conflations of a large
class of a.c.\ measures.

\bigskip

\end{abstract}


%
\allowdisplaybreaks
\section{Introduction}
{\it Conflation} is a method for consolidating a finite number of
probability distributions $P_1,\dots, P_n$ into a single probability distribution 
$Q=Q(P_1,\dots, P_n)$. The
study of this method was motivated by a basic problem in science,
namely, how best to consolidate  the information from several
independent experiments, all designed to measure the same unknown
quantity. The experiments may differ in time, geographical location,
methodology and even in underlying theory.  Ideally, of course, all
experimental data, past as well as present, should be incorporated
into the scientific record, but the
result would be of limited practical application. For many purposes, a concise 
consolidation of those distributions is more useful.

For example, to obtain the current internationally-recognized
values of each of the fundamental physical constants (Planck's
constant, Avogadro's number, etc.), the U.S.\ National Institute of
Standards and Technology (NIST) collects independent distributional
data, often assumed to be Gaussian (see Section 6), from various
laboratories. Then, for each fundamental physical constant, NIST
combines the relevant input distributions to arrive at a recommended
value and estimated standard deviation for the constant. Since
these recommended values are usually interpreted as being Gaussian,
NIST has effectively combined the several input distributions into
a single probability distribution.

The problem of combining probability distributions has been well
studied; e.g., \cite{GZ} describes a ``plethora of methods" for finding a
summary $T(P_1,\dots, P_n)$ of $n$ given (subjective) probability measures 
$P_1,\dots, P_n$ that represent
different expert opinions.  All those methods, however, including the
classical convex combination or weighted average ($T(P_1,\dots, P_n)=
\sum^n_{i=1} w_iP_i$, with nonnegative
weights $\{w_i\}$ satisfying $\sum^n_{i=1} w_i=1$) and its various 
nonlinear generalizations,
are {\it idempotent}, i.e., $T(P,\dots, P)=P$. 
For the purpose of combining probability
distributions that represent expert opinions, idempotency is a
natural requirement, since if all the opinions $P_1,\dots, P_n$ agree, the best
summary is that distribution.

But for other objectives for combining distributions, such as
consolidating the results of independent experiments, idempotency
is not always a desirable property.  Replications of the same
underlying distribution by independent laboratories, for example,
should perhaps best be summarized by a distribution with a smaller
variance. In addition to the problem of assigning and justifying
the unequal weights, another problem with the weighted averages
consolidation is that even with normally-distributed input data,
this method generally produces a multimodal distribution, whereas
one might desire the consolidated output distribution to be of
the same general form as that of the input data -- normal, or at
least unimodal.

Another natural method of consolidating distributional data -- one
that does preserve normality, and is not idempotent -- is to average
the underlying input data. In this case, the consolidation  $T(P_1,\dots, P_n)$ 
is the
distribution of $(X_1+\cdots +X_n)/n$ (or a weighted average), where  
$\{X_i\}$ are independent
with distributions $\{P_i\}$, respectively. With this consolidation method,
the variance of $T(P_1,\dots, P_n)$ is strictly smaller (unless $\{X_i\}$ 
are all constant) than the maximum variance of the $\{P_i\}$, since
$\var (P)=(\var(P_1)+\cdots + \var (P_n))/n^2$.  Input data distributions
that differ significantly, however, may sometimes reflect a higher
uncertainty or variance. More fundamentally, in general this method
requires averaging of completely dissimilar data, such as results
from completely different experimental methods (see Section~6).

The method for consolidating distributional data presented below,
called the {\it conflation} of distributions, and designated with the
symbol ``\&" to suggest consolidation of $P_1$ {\it and} $P_2$, does not require ad
hoc weights, and the mean and/or variance of the conflation may
be larger or smaller than the means or variances of the input
distributions. In general, conflation automatically gives more
weight to input distributions arising from more accurate experiments,
i.e.\ distributions with smaller standard deviations. The conflation
of several distributions  has several other properties that may
be desirable for certain consolidation objectives -- conflation
minimizes the loss of Shannon information in consolidating the
combined information from $P_1,\dots, P_n$ into a single distribution $Q$, and is
both the unique minimax likelihood ratio consolidation and the
unique proportional likelihood ratio consolidation of the given
input distributions.

In addition, conflations of normal distributions are always normal,
and coincide with the classical weighted least squares method, hence
yielding best linear unbiased and maximum likelihood estimators. Many
of the other important classical families of distributions, including
gamma, beta, uniform, exponential, Pareto, LaPlace, Bernoulli, Zeta
and geometric families, are also preserved under conflation. The
conflation of distributions has a natural heuristic and practical
interpretation --  gather data (e.g., from independent laboratories)
sequentially and simultaneously, and record the values only when
the results (nearly) agree.

\section{Basic Definition and Properties of Conflations}

 Throughout this article,
$\bbn$ will denote the natural numbers, $\bbz$ the integers, $\bbr$ the real numbers,
$(a,b]$ the half-open interval $\{x\in\bbr:a<x\le b\}$, $\bbb$ 
the Borel subsets of $\bbr$, $\calp$ the set
of all real Borel probability measures,  $\delta_x$ the Dirac delta measure in
$\calp$ at the point $x$ (i.e., $\delta_x(B)  = 1$ if $x\in B$, and 
$= 0$ if $x\notin B$), $\|\mu\|$ the total
mass of the Borel sub-probability $\mu$,  $o(~)$ the standard ``little oh"
notation $o(a_n) = b_n$  if and only if $\lim_{n\to\infty} 
\frac{a_n}{ b_n}=0$,  
{\it a.c.}\ means absolutely continuous,  the {\it p.m.f.}\ of
$P$ is the probability mass function $(p(k)=P(\{k\}))$ if $P$ is discrete and {\it p.d.f.}\ is
the probability density function (Radon-Nikodyn derivative) of $P$ if
$P$ is a.c., $E(X)$ denotes the expected value of the random variable $X$,
$\psi_P$ the characteristic function of  $P\in\calp$  (i.e., $\psi_P(t)=
\int^\infty_{-\infty}e^{itx}dP(x)$), $I_A$ is the indicator function
of the set $A$ (i.e.\ $I_A(x)=1$ if $x\in A$ and $=0$ if $x\notin A)$,
$g\otimes h$  is the convolution $(g\otimes h)(t)=\int^\infty_{-\infty}
g(t-s)h(s)ds$ of $g$ and $h$, and
$A^c$ is the complement $\bbr\backslash A$ of the set $A$. 
For brevity,  $\mu ((a,b])$ will be written $\mu(a,b]$, $\mu(\{x\})$ as
$\mu(x)$, etc.

\begin{dfn}\label{dfn2.1}  
For  $P_1,\dots,P_n\in\calp$ and $j\in \bbn$,  
$\mu_j(P_1,\dots, P_n)$ is the purely-atomic $j$-dyadic sub-probability measure
$$
\mu_j(P_1,\dots, P_n)=\sum_{k\in\bbz}\prod^n_{i=1}P_i((k-1)2^{-j},k2^{-j}]\delta_{k2^{-j}}.
$$
\end{dfn}

\remk The choice of using half-open dyadic intervals closed
on the right, and of placing the mass in every dyadic interval at the
right end point is not at all important --- the results which follow
also hold if other conventions are used, such as decimal or ternary
half-open intervals closed on the left, with masses placed at the center.

\begin{exam}\label{exttwo} 
If $P_1$ is a Bernoulli distribution with parameter  $p=\frac13$ (i.e.\
$P=\frac{(2\delta_0+\delta_1)}{ 3}$) and $P_2$ is 
Bernoulli with parameter $\frac14$, then $\mu_j(P_1,P_2)=
\frac{(6\delta_{1/2}+\delta_1)}{ 12}$  for all $j\in\bbn$.  
\end{exam}

The next proposition is the basis for
the definition of conflation of general distributions below. Recall
(e.g. \cite[Theorem~4.4.1]{C}) that for real Borel sub-probability
measures $\{\nu_j\}$ and $\nu$, the following are equivalent: 
\begin{subequations} \label{tone}
\begin{align}
\nu_j& \to\nu \hbox{ vaguely as } j\to \infty;\label{tonea}\\  
\nu_j(a,b]&\to \nu(a,b] \hbox{ for all  $a<b$ in a dense set } D\subset\bbr;
\label{toneb}\\ 
\lim_{j\to\infty}\int f(x)d\nu_j(x)&=\int f(x)d\nu(x)
\label{tonec} \\ 
 &\qquad\hbox{ 
for all continuous $f$ that vanish at infinity.}\nonumber
\end{align}
\end{subequations}

\begin{thm}\label{thmtthree}  
For all $P,P_1,\dots P_n\in\calp$
\begin{itemize}
\item[\rm (i)] $\mu_{j+1}\left(\frac{a}{ 2^m},\frac{b}{ 2^m}\right]\le
\mu_j\left(\frac{a}{ 2^m},\frac{b}{2^m}\right]$ for all 
$j, m\in\bbn$, $j>m$; and all $a\le b$, $a,b\in\bbz$;
\item[\rm (ii)] $\mu_j(P_1,\dots, P_n)$ converges vaguely to a sub-probability
measure
\item[] $\mu_\infty (P_1,\dots, P_n)$;
\item[\rm (iii)] $\lim_{j\to\infty} \|\mu_j(P_1,\dots, P_n)\|=\|\mu_\infty
(P_1,\dots, P_n)\|$; and
\item[\rm (iv)] $\mu_\infty (P)=P$,  and 
$\mu_j(P)$ converges vaguely to $P$ as $j\to\infty$.  
\end{itemize}
\end{thm}
The following simple
observation --- that the square of the sums of nonnegative numbers is always
at least as large as the sum of the squares --- will be used in the proof
of the theorem and several times in the sequel, and is recorded here
for ease of reference.  

\begin{lem}\label{lemtfour} 
For all $n\in\bbn$, all $a_{i,k}\ge 0$, and all $J\subset\bbn$,
$\prod^n_{i=1}\sum_{k\in J} a_{i,k}\ge \sum_{k\in J}\prod^n_{i=1} a_{i,k}$.
\end{lem}

\noindent {\bf Proof of Theorem~\ref{thmtthree}.}   For (i),
note that for $j>m$
\begin{subequations}\label{ttwo}
\begin{align}
\mu_j\left(\frac{a}{ 2^m},\frac{b}{ 2^m}\right]&= \sum^{b2^{j-m-1}-1}_{k=a2^{j-m-1}}\mu_j
\left(\frac{k}{ 2^{j-1}}, 
\frac{k+1}{ 2^{j-1}}\right]\label{ttwoa}\\
&=\sum^{b2^{j-m}-1}_{k=a2^{j-m}}\mu_j
\left(\frac{k}{ 2^j},\frac{k+1}{ 2^j}\right]\nonumber
\end{align}
and
\begin{equation}
\mu_{j+1}\left(\frac{a}{ 2^m}\,, \frac{b}{ 2^m}\right] = 
\sum^{b2^{j-m}-1}_{k=a2^{j-m}}\mu_{j+1}\left(\frac{k}{ 2^j}, 
\frac{k+1}{ 2^j}\right]
\label{ttwob}\end{equation}\end{subequations}  
By the definition of $\mu_j$,
\begin{subequations}\label{tthree}
\begin{align}
\mu_j\left(\frac{k}{ 2^j},\frac{k+1}{ 2^j}\right] &= \prod^n_{i=1}P_i\left(
\frac{k}{ 2^j}, \frac{k+1}{ 2^j}\right]\label{tthreea}\\
&=\prod^n_{i=1}\left(P_i\left(\frac{2k}{ 2^{j+1}},
\frac{2k+1}{ 2^{j+1}}\right]
+P_i\left(\frac{2k+1}{ 2^{j+1}}, 
\frac{2k+2}{ 2^{j+1}}\right]\right)\nonumber
\end{align}
and
\begin{align}
\mu_{j+1}\left(\frac{k}{ 2^j}, \frac{k+1}{ 2^j}\right]&=\prod^n_{i=1}P_i\left(
\frac{2k}{2^{j+1}},\frac{2k+1}{ 2^{j+1}}\right]\label{tthreeb}\\
&\qquad +\prod^n_{i=1}
P_i\left( \frac{2k+1}{ 2^{j+1}},\frac{2k+2}{ 2^{j+1}}\right].\nonumber
\end{align}\end{subequations}
By Lemma~\ref{lemtfour}, \eqref{tthreea} and \eqref{tthreeb}
imply that
\begin{equation}\label{tfour}
\mu_{j+1}\left(\frac{k}{ 2^j}, \frac{k+1}{ 2^j}\right]\le\mu_j\left(
\frac{k}{ 2^j}, \frac{k+1}{ 2^j}\right]\quad
\hbox{for all } j>m,\; j,m\in\bbn,\; k\in\bbz.\end{equation} 
By \eqref{ttwoa} and \eqref{ttwob}, this implies (i).

For (ii), note that since every sequence of sub-probability measures
contains a subsequence that converges vaguely to a sub-probability
measure (e.g. \cite[Theorem~4.3.3]{C}), there exists a subsequence 
$\{\mu_{j_k} (P_1,\dots, P_n)\}$ of $\{\mu_j(P_1,\dots, P_j)\}$
and a sub-probability measure $\mu_\infty(P_1,\dots, P_n)$ 
so that\newline
 $\mu_{j_k}(P_1,\dots, P_n)$ converges vaguely to $\mu_\infty(P_1,\dots, P_n)$
as $k\to\infty$. Hence
by the uniqueness of vague limits (i.e.\ convergence on intervals from
different dense sets results in the same limit measure 
\cite[corollary to Theorem~4.3.1, p~86]{C}), (i) implies that
$$\lim_{j\to\infty}\mu_j\left( \frac{a}{ 2^m}, \frac{b}{ 2^m}\right]=
\mu_\infty\left(\frac{a}{ 2^m}, \frac{b}{2^m}\right],$$
which proves that $\mu_j(P_1,\dots, P_n)$ converges vaguely to 
$\mu_\infty(P_1,\dots, P_n)$.  

For (iii), note that 
\begin{align*}
\lim_{j\to\infty}\|\mu_j\|&=\lim_{j\to\infty}
\sum^\infty_{k=-\infty}\mu_j(k,k+1]\\
&=
\sum^\infty_{k=-\infty}\lim_{j\to\infty}\mu_j (k,k+1]=\sum^\infty_{k=-\infty}
\mu_\infty(k,k+1]=\|\mu_\infty\|\end{align*}
where
the second equality follows by the dominated convergence theorem, and
the third by the definition of $\mu_\infty$.  The special case $n=1$ of (iv) is 
immediate.\qqed

\begin{dfn}\label{dfntfive} 
$P_1,\dots, P_n\in\calp$ are ({\it mutually}) {\it compatible}  if 
$\|\mu_j\|>0$ for all $j\in\bbn$.  
\end{dfn}

Clearly every
normal distribution is compatible with every probability distribution,
every exponential distribution is compatible with every distribution
with support in the positive reals, and every geometric distribution is
compatible with every discrete distribution having any atoms in  $\bbn$. Even
though Theorem~\ref{thmtthree}  guarantees that $\mu_j(P_1,\dots, P_n)$ 
converges vaguely to a sub-probability
measure  $\mu_\infty(P_1,\dots, P_n)$ and that
$\lim_{j\to\infty} \|\mu_j(P_1,\dots, P_n\| = \|\mu_\infty(P_1,\dots, P_n)\|$, 
and compatibility implies that 
$\frac{\mu_j(P_1,\dots, P_n)}{\|\mu_j(P_1,\dots, P_n)\|}$
is a probability measure for all $j\in\bbn$, 
$\lim_{j\to\infty}\frac{\mu_j (P_1,\dots, P_n)}{\|\mu_j(P_1,\dots, P_n)\|}$
may not be a probability measure, as the next example shows.

\begin{exam}\label{extsix} 
Let  $P_1=\sum_{k\in N}2^{-k}\delta_k$, and
$P_2=\sum_{k\in\bbn} 2^{-k}\delta_{k+2^{-k}}$.
Then $P_1$ and $P_2$ are easily seen to be compatible, but 
$\lim_{j\to\infty}\frac{\mu_j (P_1,\dots, P_n)}{
\|\mu_j(P_1,\dots, P_n)\|}$ is the
zero measure, since for each $j\in\bbn$, the support of the probability measure 
$\frac{\mu_j(P_1,\dots, P_n)}{ \|\mu_j(P_1,\dots, P_n)\|}$ is
contained in $[j,\infty)$.  
\end{exam}

The next definition is the main definition in this paper.

\begin{dfn}\label{dfntseven}  
If  $\frac{\mu_j(P_1,\dots, P_n)}{\|\mu_j(P_1,\dots, P_n)\|}$ 
converges vaguely to a Borel probability measure $Q$
as $j\to\infty$, this limit $Q$ is called the {\it conflation} of $P_1,\dots, P_n$,
written $\&(P_1,\dots, P_n)$.
\end{dfn}

\begin{thm}\label{thmteight}  
The operation $\&$
is commutative and associative, that is, $\&(P_1,P_2)=\&(P_2,P_1)$ and 
$\&(P_1,\&(P_2,P_3))=\&(\&(P_1,P_2),P_3) =$\newline $\&(P_1,P_2,P_3)$.  
\end{thm}

\begin{proof} Immediate from
the definition of $\mu_\infty$ since multiplication of real numbers is commutative
and associative.\end{proof}

\begin{exam}\label{extnine} 
Let $P_1$ be a Bernoulli distribution with parameter $p=\frac13$ and $P_2$ be
Bernoulli with parameter $\frac14$, as in Example~\ref{exttwo}. Then 
$\&(P_1,P_2)=\frac{(6\delta_0+\delta_1)}{ 7}$.
\end{exam}

\begin{exam}\label{extten} 
Let $P_1$ be $N(0,1)$ and $P_2$ be Bernoulli with parameter $p=\frac13$. 
Then it can easily be seen that  $\frac{\mu_j(P_1,P_2)}{\|\mu_j(P_1,P_2)\|}$
converges vaguely to 
$\& (P_1,P_2)=\delta_0 \left(\frac{2}{ (2+e^{-1/2})}\right)+\delta_1
\left( \frac{e^{-1/2}}{ (2+e^{-1/2})}\right)$, that is, to
the probability measure having the same atoms as the discrete measure,
weighted according to the product of the atom masses of $P_2$ and the magnitude
of the density of $P_1$ at $0$ and $1$.
\end{exam}

\section{Conflations of Discrete and of Absolutely Continuous Distributions}

In general, explicit representations of conflations are not known in
closed form. For large natural classes of distributions, however, such as
collections of discrete distributions with common atoms and collections
of a.c.\ distributions with overlapping densities, explicit forms of
the conflations are easy to obtain. The next two theorems give simple
and powerful characterizations of conflations in those two cases. Since
in practice input data can easily be approximated extremely closely by
discrete distributions with common atoms (e.g., by replacing each $P_i$  by the
dyadic approximation $\mu_j(P_i)$ above), or can be smoothed (e.g.\ by convolution with
a $U(-\ep,\ep)$ or a $N(0,\ep^2)$ variable), 
these two cases are of practical interest. The third
conclusion in the next two theorems also yield the heuristic and useful
interpretation of conflation described in the introduction.

\begin{thm}\label{thmthone} 
Let $P_1,\dots, P_n$ be discrete with p.m.f.'s $p_1,\dots, p_n$, respectively, and common
atoms $A$, where $\emptyset\ne A\subset\bbr$.  Then $\&(P_1,\dots, P_n)$ exists, and the
following are equivalent:
\begin{itemize}
\item[\rm (i)] $Q=\&(P_1,\dots, P_n)$
\item[\rm (ii)] $Q=\frac{\sum_{x\in A}\delta_x\prod^n_{i=1}p_i(x)}{
\sum_{y\in A}\prod^n_{i=1} p_i(y)}$
\item[\rm (iii)] $Q$ is the conditional distribution of $X_1$ given that $X_1=X_2=\cdots=
X_n$, where $X_1,\dots, X_n$ are independent r.v.'s with distributions $P_1,\dots, P_n$,
respectively.
\end{itemize}
\end{thm}

\begin{proof} Fix $P_1,\dots, P_n$ and note that by definition of atom, 
$p_i(x)>0$ for all
$i=1,\dots, n$ and all $x\in A$.  Fix $k_0\in\bbz$ and $j_0\in\bbn$, and let
$D=\left( \frac{k_0}{ 2^{j_0}}, \frac{k_0+1}{ 2^{j_0}}\right]$.  First it will be shown that
\begin{equation}\label{thone}
\mu_\infty(D)=\sum_{x\in A\cap D}\prod^n_{i=1} p_i(x).
\end{equation}  
For all $x\in\bbr$, $j\in\bbn$, let $D_{x,j}$ denote the unique dyadic interval 
$\left(\frac{k}{ 2^j}, \frac{k+1}{ 2^j}\right]$ containing $x$.  Note that
$D_{x,j}\searrow \{x\}$ 
as $j\to\infty$ so $P_i(D_{x,j})\searrow p_i(x)$ as $j\to\infty$ for all 
$i$ and all $x\in\bbr$.  

This implies
\begin{equation}\label{thtwo}
\lim_{j\to\infty}\prod^n_{i=1}P_i(D_{x,j})=\prod^n_{i=1} p_i(x)\quad
\hbox{ for all } x\in\bbr.\end{equation}    
Fix $\ep>0$.   Since $\{P_i\}$ are discrete, there exists a finite set 
$A_0\subset \bbr$ such that
\begin{equation}\label{ththree}
P_i(D\cap A^c_0)<\ep\qquad\hbox{for all }i\in\{1,\dots, n\}.
\end{equation}
Since $\prod^n_{i=1}p_i(x)=0$ for all $x\in A^c$, \eqref{ththree}  implies
\begin{align}\label{thfour}
\Bigl| \sum_{x\in A\cap D}\prod^n_{i=1}p_i(x)&-\sum_{x\in A_0\cap D}
\prod^n_{i=1} p_i(x)\Bigr| =
\sum_{x\in A\cap A^c_0\cap D}\prod^n_{i=1}p_i(x)\\
 &\le
\sum_{x\in A\cap A^c_0\cap D} p_1(x)
\le P_1(D\cap A^c_0)<\ep.\nonumber\end{align}
For each $j\in\bbn$, let $S_j=\bigcup_{x\in A_0}D_{x,j}$.  Then since $x\in D_{x,j}$ for
all $x$ and $j$, \eqref{ththree} implies 
$P_i(D\cap S^c_j)<\ep$ for all $i\in\{1,\dots, n\}$. 
Thus by definition of $\{\mu_j\}$ and Lemma~\ref{lemtfour}, 
\begin{equation}\label{thfive}
\mu_j(D\cap S^c_j)\le \prod^n_{i=1} P_i(D\cap S^c_j)<\ep^n\qquad
\hbox{for all }j\in\bbn. \end{equation} 
This implies that
\begin{align}\label{thsix}
\mu_j(D)&=\mu_j(D\cap S_j)+\mu_j(D\cap S^c_j)\\
&=\sum_{x\in D\cap A_0}
\mu_j(D_{x,j})+\mu_j(D\cap S^c_j)\nonumber\\
&= \sum_{x\in D\cap A_0}\prod^n_{i=1}P_i(D_{x,j})+\mu_j(D\cap S^c_j)\nonumber
\end{align}
where the second equality follows from the definitions of $S_j$ and
$D_{x,j}$. Since $x\in D_{x,j}$, \eqref{thsix}
implies
\begin{equation}\label{thseven}
\mu_j(D)\ge\sum_{x\in D\cap A_0}\prod^n_{i=1} P_i(D_{x,j})\ge\sum_{x\in D\cap A_0}
\prod^n_{i=1} p_i(x).\end{equation}
By \eqref{thsix}, \eqref{thtwo}  and \eqref{thfive},
\begin{equation}\label{theight}
\mu_j(D)\le \sum_{x\in A_0\cap D}\prod^n_{i=1} p_i(x)+\ep^n
+\ep\quad\hbox{for sufficiently large } j.\end{equation}
By \eqref{thseven} and \eqref{theight}, 
$|\mu_j(D)=\sum_{x\in A_0\cap D}\prod^n_{i=1}p_i(x)|
\le\ep+\ep^n$, so by \eqref{thfour},\newline
 $|\mu_j(D)-\sum_{x\in A\cap D}\prod^n_{i=1}
p_i(x)|<2\ep+\ep^n$.  Since $\ep >0$ was arbitrary and since $\mu_j\to \mu_\infty$, this
implies \eqref{thone}. Since $D$ was arbitrary, \eqref{thone} implies that 
$\|\mu_\infty(P_1,\dots, P_n)\|=\sum_{x\in A}\prod^n_{i=1}p_i(x)$, which proves
that $\&(P_1,\dots, P_n)$ exists. The equivalence of (i) and (ii) 
follows since  
$\&(P_1,\dots, P_n)=\frac{\mu_\infty}{\|\mu_\infty\|}$ and since
the measures of dyadic intervals  $D$ determine $\mu_\infty$. The equivalence of (ii) and
(iii) follows immediately from the definition of conditional probability.
\end{proof}

\begin{exam}\label{exthtwo} 
If $P_1$ is binomial with parameters $n=2$ and $p=\frac13$ and $P_2$ is Poisson with
parameter $\lam=5$,
 then $\&(P_1,P_2)$ is discrete
with atoms only at 0, 1  and 2 --- specifically, 
$\&(P_1,P_2)=\frac{8\delta_0}{ 73}
+\frac{40\delta_1}{ 73} + \frac{25\delta_2}{ 73}$.  
\end{exam}

\remk It should be noted that if the input distributions are
discrete and have no common atoms, then the conflation does not
exist. This could happen if, for example, the underlying experiments
were designed to estimate Avogadro's number (theoretically a 24-digit
integer), and the results were given as exact integers. In practice,
however, Avogadro's number is known only to seven decimal places,
and if the results of the experiments were reported or recorded to
eight or nine decimal places of accuracy, then there would almost
certainly be common values, and the conflation would be well defined
and meaningful. (Restriction to the desired decimal accuracy could
be done by the experimenter, or afterwards, e.g.\ converting each
input $P_i$ to $\mu_{20}(P_i)$ as mentioned above.)

The analog
of Theorem~\ref{thmthone}
 for probability distributions with densities requires an
additional hypothesis on the density functions, for the simple reason that
the product of a finite number of p.m.f.'s is always the mass function of
a discrete sub-probability measure (i.e., is always summable), but the
product of a finite number of p.d.f.'s may not be the density function
of a finite a.c.\ measure (i.e., may not be integrable), as will be seen
in Example~\ref{exthsix} below. 

The algebraic and Hilbert space properties of
normalized products of density functions have been studied for special
classes of a.c. distributions with p.d.f.'s with compact support that
are bounded from above and bounded from below away from zero \cite{Aitc,
EDP}; products of p.m.f.'s and p.d.f.'s have been used in certain
pattern-recognition problems \cite{Lu}; and the ``log opinion poll" method for
combining probability distributions \cite{GZ} is an a.c.\ distribution with 
normalized density $\prod f^{w_i}_i$, 
which is similar in structure, but is idempotent since the weights sum to one.

\begin{thm}\label{thmththree}  
Let $P_1,P_2,\dots, P_n$ be absolutely continuous with densities
$f_1,\dots, f_n$ satisfying $0<\int^\infty_{-\infty}\prod^n_{i=1} f_i(x)dx<\infty$.
Then $\&(P_1,\dots, P_n)$ exists and the following are equivalent:
\begin{itemize}
\item[\rm (i)] $Q=\&(P_1,\dots, P_n)$;
\item[\rm (ii)] $Q$ is absolutely continuous with density
$f(x)=\frac{\prod^n_{i=1}f_i(x)dx}{\int^\infty_{-\infty}\prod^n_{i=1}f_i(y) dy}$;
\item[\rm (iii)] $Q$ is the (vague) limit, as $\ep\searrow 0$, of the conditional
distribution of $X_1$ given that $|X_i-X_j|<\ep$ for all $i,j\in \{1,\dots, n\}$,
where $X_1,\dots, X_n$ are independent r.v.'s with distributions $P_1,\dots, P_n$,
respectively.
\end{itemize}
\end{thm}


\begin{proof}  First suppose that the densities $\{f_i\}$ are nonnegative
simple functions on half-open dyadic intervals	$(a,b]$, 
$a,b\in \{\frac{k}{ 2^j}\,: k\in \bbz,j\in\bbn\}$.  Without loss of generality
(splitting the intervals if necessary), there exists $j_0\in\bbn$  and a finite set
$K\subset \bbn$ such that
\begin{equation}\label{thnine}
f_i=\sum_{k\in K}c_{j,k}I_{D_k}\quad\hbox{for all } i=1,\dots, n
\end{equation}
where $c_{i,k}\ge 0$ for all $i,k;$ and $D_k$ are disjoint intervals
$\left(a_k,a_k+\frac{1}{ 2^{j_0}}\right]$, 
$a_k=\frac{k}{ 2^{j_0}}$, $k\in K$.  
Let $\pi_k=\prod^n_{i=1}c_{i,k}$ for all $k\in K$, and note that
the compatibility of $P_1,\dots, P_n$ implies that $\sum_{k\in K}\pi_k>0$.
It will now be shown that $\&(P_1,\dots, P_n)$ is absolutely continuous with
density $f$, where
\begin{equation}\label{thten}
f(x)=\frac{\prod^n_{i=1}f_i(x)}{\int^\infty_{-\infty}\prod^n_{i=1} f_i(s)ds}\, 
=\frac{2^{j_0}\sum_{k\in K}\pi_kI_{D_k}}{\sum_{k\in K}\pi_k}\qquad
\hbox{ a.s.} \end{equation} 
Fix $m\in \bbn$, and let $a_{k,s}=a_k+\frac{s}{ 2^{j_0+m}}$.  First note since
$f_i=c_{i,k}$ a.s.\ on $D_k$ for each $i$ and $k$,
\begin{align}\label{theleven}
\pi_k&=(2^{j_0+m})^n\prod^n_{i=1}P_i (a_{k,s-1},a_{k,s}]\\
&\qquad\qquad \hbox{for all }
s=1,\dots, 2^m;\; m\in \bbn;\; \hbox{ and } k\in K.\nonumber
\end{align} 
By \eqref{theleven}, and the definitions of $\{D_k\}$ and $\{\mu_j\}$,
\begin{align}\label{thtwelve}
\mu_{j_0+m}&=\sum_{k\in K}\sum^{2^m}_{s=1}\prod^n_{i=1}P_i(a_{k,s-1},a_{k,s}]
\delta_{a_{k,s}}\\
& = \sum_{k\in K}\sum^{2^m}_{s=1}\frac{\pi_k}{ 2^{(j_0+m)n}}
\delta_{a_{k,s}}=\frac{1}{ 2^{(j_0+m)n}}\sum_{k\in K}\pi_k
\sum^{2^m}_{s=1}\delta_{a_{k,s}}.\nonumber\end{align}
Since $m$, $j_0$ and $n$ are fixed, and since 
$\|\sum^{2^m}_{s=1}\delta_{a_{k,s}}\|=2^m$, \eqref{thtwelve}  implies that
\begin{equation}\label{ththirteen}
\frac{\mu_{j_0+m}}{\|\mu_{j_0 +m}\| }=
\frac{\sum_{k\in K}\pi_k\sum^{2^m}_{s=1}\delta_{a_{k,s}}}{
2^m\sum_{k\in K}\pi_k}=
\frac{\sum_{k\in K} \pi_k \frac{1}{ 2^m}\sum^{2^m}_{s=1}\delta_{a_{k,s}}}{
\sum_{k\in K}\pi_k}\end{equation}
But since $\frac{1}{ 2^m}\sum^{2^m}_{s=1}\delta_{a_{k,s}}$  converges vaguely to the
probability measure uniformly distributed on $D_k$ for each $k\in K$, and
$\frac{\mu_{j_0+m}}{ \|\mu_{j_0+m}\|}$ converges vaguely to $\& (P_1,\dots, P_n)$ as
$m\to\infty$, 
\eqref{ththirteen} implies \eqref{thten}. 
This completes the proof that $\&(P_1,\dots, P_n)$ exists and (i) and (ii) are equivalent
when the densities are simple functions on dyadic intervals. For
the general case, use the standard method to extend this
result to general simple functions, and then, since densities
are a.s. nonnegative, extend this to finite collections of
densities whose product is integrable, via the standard argument
of approximating below by simple functions, and using monotone
convergence.  

For the equivalence of (ii) and (iii), for every $\ep>0$
let $P_{1,\ep}$ denote the conditional distribution of $X_1$ given 
$\{|X_i-X_j|<\ep$ for all $i,j\in \{1,\dots, n\}\}$, that is, for all
Borel sets $A$,
$$P_{1,\ep}(A)=\frac{P(X_1\in A\hbox{ and } |X_i-X_j|<\ep\hbox{ for all } i,j\in \{1,\dots, n\})
}{ P(|X_i-X_j|<\ep\hbox{ for all } i,j\in \{1,\dots, n\})}\,,$$
where the denominator is always strictly positive since by hypothesis
\newline $\int\prod^n_{i=1}f_i(x)dx>0$.
Clearly, $P_{1,\ep}$ is absolutely continuous with conditional density $f_{1,\ep}$,
where the independence of the $\{X_i\}$ implies that
\begin{equation}\label{thfourteen}
f_{1,\ep}(x)=\frac{f_1(x)\left(\prod^n_{i=2}\int^{x+\ep}_{x-\ep}f_i(z)dz\right)
}{
\int^\infty_{-\infty}f_1(y)\left(\prod^n_{i=2}\int^{y+\ep}_{y-\ep}
f_i(z)dz\right)dy}.\end{equation}
Next note that by the definition of derivative and integral,
\begin{equation}\label{thfifteen}
\lim_{\ep \to 0} f_1(x)\prod^n_{i=2}(2\ep)^{-1}\int^{x+\ep}_{x-\ep}
f_i(z)dz=\prod^n_{i=1}f_i(x).\end{equation}
Letting $f^M_i=\min\{f_i,M\}$ for all $M\in\bbn$, and all $i\in \{1,\dots, n\}$, calculate
\begin{align}\label{thsixteen}
&\lim_{\ep \to 0} \int f_1(y)\left(\prod^n_{i=2}(2\ep)^{-1}\int^{y+\ep}_{y-\ep}
f_i(z)dz\right)dy\\
&\qquad =\lim_{\ep\to 0}\lim_{M\to \infty}\int f^M_1 (y)
\left((2\ep)^{-1}\prod^n_{i=2}\int^{y+\ep}_{y-\ep}f^M_i(z)dz\right)dy\nonumber\\
&\qquad =\lim_{M\to\infty}\lim_{\ep\to 0}\int f^M_1(y)\left(\prod^n_{i=2}(2\ep)^{-1}
\int^{y+\ep}_{y-\ep} f^M_i(z)dz\right)dy\nonumber\\
&\qquad =
\lim_{M\to\infty}\int \prod^n_{i=1}f^M_i(y)dy=\int\prod^n_{i=1}
f_i(y)dy,\nonumber\end{align}
where the first equality follows from the monotone convergence
theorem, the second since the convergence of 
$\lim_{\ep\to 0} \int f^M_1 (y)\left(\prod^n_{i=2}(2\ep)^{-1}
\int^{y+\ep}_{y-\ep} f^M_i(z)dz\right)dy$ is uniform in $M$, the third
by \eqref{thfifteen}  and the bounded convergence theorem since the integrand
is bounded by $M^n$, and the last by the dominated convergence theorem
since by hypothesis,  
$$\int^\infty_{-\infty}\prod^n_{i=1} f_i(x)dx<\infty.$$
Thus by \eqref{thfourteen}, \eqref{thfifteen}, and \eqref{thsixteen}, 
$$\lim_{\ep \to 0} f_{1,\ep}(x)=\frac{\prod^n_{i=1}f_i(x)}{ \int\prod^n_{i=1}
f_i(y)dy},$$ 
proving the equivalence of (ii) and (iii).\end{proof}

\begin{exam}\label{exthfour} 
Suppose  $P_1$ is $N(0,1)$ and $P_2$ is exponentially distributed with mean 1.
Then $\&(P_1,P_2)$ is a.c.\ with p.d.f.\ $f(x)$ proportional to
$e^{-x^2/2}e^{-x}=e^{1/2}e^{-(x+1)^2/2}$ for $x>0$, 
which is simply the standard
 normal shifted to the left one unit, and conditioned to be nonnegative.
\end{exam}

\begin{exam}\label{exthfive} 
Suppose $P_1$ and $P_2$ are both standard Cauchy distributions. Then
neither $P_1$ nor $P_2$ have finite means, but by Theorem~\ref{thmththree}, 
$\&(P_1,P_2)$ is a.c.\ with
density $f(x)=c(1+x^2)^{-2}$ for some $c>0$, 
and since  $\int^\infty_{-\infty} x^2(1+x^2)^{-2}dx<\infty$,
$\&(P_1,P_2)$ has both finite mean and variance. In particular,
the conflation of Cauchy distributions is not Cauchy, in contrast to
the closure of many classical families under conflation 
(Theorem~\ref{thmsvone}
below). This example also shows that the classes of stable and infinitely
divisible  distributions are not closed under conflation.
\end{exam}

In general, the conflation of a.c.\ distributions, even an
a.c.\ distribution with itself, may not be a.c., let alone have a density
proportional to the product of the densities.

\begin{exam}\label{exthsix} 
Let $P_1=P_2$ be a.c.\ with p.d.f.\ $f(x)=(4x)^{-1/2}$ for
$0<x<1$ (and zero elsewhere). Then $f_1(x)f_1(x)=\frac1{4x}$ is not
integrable, and no scalar multiple is a p.d.f.  However, the conflation
$\&(P_1,P_2)$ does exist, and by showing that the normalized mass of $\mu_j$ is moving to
the left as  $j\to\infty$ it can be seen that 
$\&(P_1,P_2)=\delta_0$, the Dirac delta measure at zero
(in particular, the conflation is not even a.c.).
\end{exam}

The characterization of the conflation of a.c. distributions as
the normalized product of the density functions yields another
characterization of conflations of a.c. distributions, an analog of the
classical convolution theorem in Fourier analysis \cite{B}.

Recall that $g\otimes h$ is the convolution of $g$ and $h$.

\begin{thm}[Convolution theorem for conflations]\label{thmthseven} 
Let $P_1,P_2,\dots, P_n$ be compatible and a.c.\ with densities $\{f_i\}$ and
characteristic functions $\{\psi_i\}$.  If \/
$0<\int^\infty_{-\infty}\prod^n_{i=1}
f_i (x)dx<\infty$ and $\{\psi_i\}$ are $L^1$, then $\& (P_1,\dots, P_n)$ exists
and is the unique a.c.\ probability distribution with characteristic function
$\psi_{\&(P_1,\dots, P_n)}=\frac{\psi_1\otimes\psi_2\otimes\cdots\otimes \psi_n}{
 (2\pi)^{n-1}\int^\infty_{-\infty}\prod^n_{i=1} f_i(x)dx}$.
\end{thm}

\begin{proof}  The proof
will be given only for the case $n=2$; the general case follows easily by
induction and Theorem~\ref{thmteight}. 
Suppose $\psi_1$ and $\psi_2$ are $L^1$ and 
$0<\int^\infty_{-\infty} f_1(x)f_2(x)dx<\infty$.  Then
\begin{align*}
(\psi_1\otimes\psi_2)(t)&=\int^\infty_{-\infty}\psi_2(s)\psi_1(t-s)ds=
\int^\infty_{-\infty}\psi_2(s)\left[\int^\infty_{-\infty}e^{i(t-s)x}f_1(x)dx\right]ds\\
&=\int^\infty_{-\infty} f_1(x)e^{itx}\left[\int^\infty_{-\infty}
\psi_2(s)e^{-isx}ds\right]dx\\
&= \int^\infty_{-\infty} 2\pi f_1(x)f_2(x)e^{itx}dx=2\pi\psi_{\&(P_1,P_2)}(t)
\int^\infty_{-\infty} f_1(x)f_2(x)dx\end{align*}
where the first equality follows from the definition of convolution;
the second by definition of  $\psi_1$; the third by Fubini's theorem since 
$\psi_1$ and $\psi_2$ are
absolutely integrable; the fourth by the inverse characteristic function
theorem (e.g.\ \cite[Theorem~6.2.3]{C}) since $\psi_2$ is $L^1$; 
and the last equality by Theorem~\ref{thmththree} since 
$0<\int^\infty_{-\infty}f_1(x)f_2(x)dx<\infty$.\end{proof}

The next example is an application of Theorem~\ref{thmthseven}, and shows
that the conflation of two standard normal distributions is mean-zero
normal with half the variance of the standard normal. An intuitive
interpretation of this fact is that if the two standard normals reflect
the results of two independent experiments, then combining these results
effectively doubles the number of trials, thereby halving the variance
of the (sample) means. Normality is always preserved under conflation,
as will be seen in Theorem~\ref{thmsvone} below.

\begin{exam}\label{extheight} 
Let  $P_1=P_2$ be $N(0,1)$, so $\psi_1(t)=\psi_2(t)=e^{-t^2/2}$.  
Then 
$(\psi_1\otimes\psi_2)(t)=\int^\infty_{-\infty}e^{-(t-s)^2/2} 
e^{-s^2/2}ds=e^{-t^2/4}\int^\infty_{-\infty}
e^{-\left(s-\frac{t}{ 2}\right)^2}ds=e^{-t^2/4}\sqrt\pi$, so since
 \newline
$\int f_1(x)f_2(x)dx=\int^\infty_{-\infty}\frac{e^{-x^2/2}}{ \sqrt{2\pi}}
\frac{e^{-x^2/2}}{\sqrt{2\pi}}\, dx=\frac{1}{ 2\sqrt\pi}$,
Theorem~\ref{thmthseven} implies that $\&(P_1,P_2)$ is a.c. with 
characteristic function $\psi(t)=\frac{\sqrt\pi e^{-t^2/4}}{ 2\pi/2\sqrt\pi}=
e^{-t^2/4}$, so  $\&(P_1,P_2)$ is $N(0,\frac12)$.
\end{exam}

In general, the convolution of characteristic functions of discrete
measures may not even exist.  

\begin{exam}\label{exthnine} 
Let $P=P_1=P_2=\delta_0$.  Then it is easy to see that
$\&(P_1,P_2)=\delta_0$, and $\psi_P(t)\equiv 1$, so $\psi_{P_1}\otimes
\psi_{P_2}$ does not even exist.
\end{exam}

\section{Minimal Loss of Shannon Information}

Replacing several distributions by a single distribution will 
clearly result in some loss of information, however that is defined. 
A classical measure of information in a stochastic setting is the 
Shannon Information.

Recall that the {\it Shannon Information} $S_P(A)$
(also called the {\it surprisal}, or {\it self-information}) of  
a probability
$P$ for the event
$A\in\bbb$, is $S_P(A)=-\log_2P(A)$ (so the smaller the value of
$P(A)$, the greater the information or
surprise). The {\it information entropy}, which will not be addressed here,
is simply the expected value of the Shannon information.  

\begin{exam}\label{exfone}  
If $P$ is uniformly distributed on $(0,1)$, and $A=(0,\frac14)\cup(\frac12,\frac34)$,
then $P(A)=\frac12$,  so $S_P(A)=-\log_2(P(A))=1$. 
 Thus exactly one bit of information is obtained by observing $A$, namely,
that the value of the second binary digit is 0.
\end{exam}

\begin{dfn}\label{dfnftwo} 
The ({\it joint}) {\it Shannon Information of  $P_1, P_2,\dots, P_n$  for
the event} $A\in\bbb$, is
$$S_{\{P_1,\dots, P_n\}}(A)=S_P(X_1\in A,\dots, X_n\in A)=
\sum^n_{i=1}S_{P_i}(A)=-\log_2\prod^n_{i=1}
P_i(A)$$
where $\{X_i\}$ are independent random variables with distributions
$\{P_i\}$, respectively, 
and the {\it loss between the Shannon Information of $Q \in \calp$ and  
$P_1,\dots, P_n$ for the event}
$A\in \bbb$ is  
$S_{\{P_1,\dots, P_n\}}(A)-S_Q(A)$  if 
$\prod^n_{i=1}P_i(A)>0$,  and is 0 if 
$Q(A)=\prod^n_{i=1}P_i(A)=0$.  
\end{dfn}

Note that the maximum loss is always
non-negative (taking $A=\uom$).

The next theorem characterizes
conflation as the minimizer of loss of Shannon Information.  

\begin{thm}\label{thmfthree} 
If  $P_1,\dots, P_n\in\calp$  satisfies $\|\mu_\infty(P_1,P_2,\dots, P_n)\|>0$, 
then 
\begin{itemize}
\item[\rm (i)] the conflation $\& (P_1,P_2,\dots, P_n)$ exists; 
\item[\rm (ii)] for every
$Q\in\calp$, the maximum loss between the Shannon Information of  $Q$ and    
$P_1,\dots, P_n$ is at
least $\log_2(\|\mu_\infty(P_1,P_2,\dots, P_n)\|^{-1})$;
and 
\item[\rm (iii)] the bound in (ii) is attained if and only if 
$Q = \&(P_1, P_2,\dots, P_n)$.  
\end{itemize}
\end{thm}

\begin{proof}  Fix $P_1,\dots,P_n\in\calp$, and for brevity, let $\mu_j=\mu_j
(P_1,P_2,\dots, P_n)$  for all $j\in \bbn$, and $\mu_\infty = \mu_\infty
(P_1,P_2,\dots, P_n)$.  For (i), note that
by Theorem~\ref{thmtthree}, $\mu_j$ converges vaguely to $\mu_\infty$, and  
$\lim_{j\to\infty}\|\mu_j\|=\|\mu_\infty\| >0$, so $\mu_j\|\mu_j\|^{-1}$ converges vaguely to
the probability measure $\mu_\infty\|\mu_\infty\|^{-1}$, 
which implies that $\&(P_1,P_2,\dots, P_n)$ exists.  

For (ii) and (iii),
fix $Q\in \calp$, and let $\& =\&(P_1,P_2,\dots, P_n)$. It must be shown that
\begin{subequations}\label{fone}
\begin{align}
S_{\{P_1,\dots, P_n\}}(A)-S_Q(A)&\ge \log_2(\|\mu_\infty\|^{-1}) \hbox{ for some Borel }
A\label{fonea}\\
S_{\{P_1,\dots, P_n\}}(A)-S_Q(A)&> \log_2(\|\mu_\infty\|^{-1}) \hbox{ for some Borel }
A \hbox{ if } Q\ne\& \label{foneb}\\
\noalign{\noindent and}
S_{\{P_1,\dots, P_n\}}(A)-S_Q(A)&\le \log_2(\|\mu_\infty\|^{-1}) \hbox{ for all Borel }
A \hbox{ if } Q=\&. \label{fonec}\end{align}\end{subequations}
By definition of Shannon Information, and since $\log_2 (x)$ is increasing,
\eqref{fonea}--\eqref{fonec} are equivalent to
\begin{subequations}\label{ftwo}
\begin{align}
\frac{Q(A)}{\prod^n_{i=1}P_i(A)} & \ge \|\mu_\infty\|^{-1}\hbox{ for some Borel } A
\label{ftwoa}\\
\frac{Q(A)}{\prod^n_{i=1}P_i(A)} & > \|\mu_\infty\|^{-1}\hbox{ for some Borel } A \hbox{ if }
Q\ne\&  \label{ftwob}\\
\frac{Q(A)}{\prod^n_{i=1}P_i(A)} & \le \|\mu_\infty\|^{-1}\hbox{ for all Borel } A \hbox{ if }
Q=\&.  \label{ftwoc}\end{align}\end{subequations}
To establish \eqref{ftwoa}, fix $\ep$, $\|\mu_\infty\|^{-1}>\ep >0$.  
By Theorem~\ref{thmtthree}, $\|\mu_j\|\to\|\mu_\infty\|$ as $j\to\infty$, 
so there exists $j^*\in\bbn$ such that
\begin{equation}\label{fthree}
\|\mu_{j^*}\|^{-1}> \|\mu_\infty\|^{-1}-\ep >0.\end{equation}
For each $k\in\bbz$, let $q_k=Q\left(\frac{k}{ 2^{j^*}},\frac{k+1}{ 2^{j^*}}\right]$, and 
$p_k=\prod^n_{i=1} P_i\left(\frac{k}{ 2^{j^*}}, \frac{k+1}{ 2^{j^*}}\right]$,
note that by the definition of $\{\mu_j\}$,
\begin{equation}\label{ffour}
\|\mu_{j^*}\|=\sum_{k\in\bbz} p_k.\end{equation}
By \eqref{fthree}, since  $Q$ is a probability, \eqref{ffour} implies that 
$1=\sum_{k\in\bbz}q_k=$\newline
 $\sum_{k\in\bbz} p_k\|\mu_{j^*}\|^{-1}$, so there exists
$k^*\in\bbz$ such that
\begin{equation}\label{ffive}
q_{k^*}\ge p_{k^*}\|\mu_{j^*}\|^{-1}>0.\end{equation}
Hence, by \eqref{fthree} and \eqref{ffive}
  and the definition of $\{p_k\}$ and $\{q_k\}$,
\begin{equation}\label{fsix}
\frac{Q\left( \frac{k^*}{ 2^{j^*}}, \frac{k^*+1}{ 2^{j^*}}\right]}{
\prod^n_{i=1} P_i \left( \frac{k^*}{ 2^{j^*}}, \frac{k^*+1}{ 2^{j^*}}\right]}
\ge \|\mu_\infty\|^{-1}-\ep.\end{equation}
By Lyapounov's theorem, the range of a finite-dimensional vector measure
is closed (e.g.\ \cite{Ly} or \cite[Theorem~1.1]{EH}), so since  $\ep$ was arbitrarily
small, this proves \eqref{ftwoa}.

	To prove \eqref{ftwob}, suppose $Q\ne\&$. Then there exists a $c>0$, 
$k^*\in\bbz$ and $j^*\in \bbn$, such that for $D=\left( \frac{k^*}{ 2^{j^*}},
\frac{k^*+1}{ 2^{j^*}}\right]$, $\& (D)>0$ and $Q(D)>\&(D)+c\mu_\infty(D)$.
Since $\&=\frac{\mu_\infty}{\|\mu_\infty\|}$, this
	implies that 
\begin{equation}\label{fseven}
\frac{Q(D)}{ \mu_\infty(D)} > \|\mu_\infty\|^{-1}+c.\end{equation}
Since $\mu_j(D)\to \mu_\infty(D)$ as $j\to\infty$ by 
Theorem~\ref{thmtthree}(ii), 
\eqref{fseven} implies there exists an  $m\in \bbn$ so that
\begin{equation}\label{feight}
\frac{Q(D)}{\mu_{j^*+m} (D)}>\|\mu_\infty\|^{-1}+\frac{c}{ 2}\,.
\end{equation}
Note that 
$D=\bigcup_{k\in J} D_k$, where $D_k=\left( \frac{k}{ 2^{j^*+m}},
\frac{k+1}{ 2^{j^*+m}}\right]$
and $J=\{k^*2^m, k^*2^m+1,\dots, k^*2^m+2^m-1\}$.
Next, note that since $\frac{\sum a_k}{\sum b_k}\le \max_k
\left\{ \frac{a_k}{b_k}\right\}$
for nonnegative $\{a_k,b_k\}$, there exists $M\in J$ such that
\begin{equation}\label{fnine}
\frac{\sum_{k\in J} Q(D_k)}{\sum_{k\in J}\prod^n_{i=1} P_i(D_k)}\le
\max_{k\in J}\frac{Q(D_k)}{\prod^n_{i=1}P_i(D_k)} = 
\frac{Q(D_M)}{\prod^n_{i=1} P_i(D_M)}.\end{equation}
Then
\begin{equation}\label{ften}
\frac{Q(D_M)}{\prod^n_{i=1} P_i(D_M)}\ge
\frac{Q(D)}{\mu_{j^*+m}(D)} > \|\mu_\infty\|^{-1}\end{equation}
where the first inequality in \eqref{ften}  follows by 
\eqref{fnine}  and since $\mu_{j^*+m}(D)=
\sum_{k\in J}\prod^n_{i=1} P_i(D_k)$, and
the second by \eqref{feight}. This proves \eqref{ftwob}. 
Finally, suppose that $\&=Q$. Since
the class of sets $\left\{ \left(\frac{k}{ 2^j}\,, 
\frac{k+1}{ 2^j}\right]: j\in\bbn , k\in\bbz\right\}$
generates the Borel sigma algebra on $\bbr$, and since $Q=\&=\mu_\infty
\|\mu_\infty\|^{-1}$,
to prove \eqref{ftwoc} it is enough to show that for all $j\in\bbn$,
all finite sets $J\subset \bbn$ and all $D=\bigcup_{k\in J} \left(
\frac{k}{ 2^j},
\frac{k+1}{ 2^j}\right]$,
\begin{equation}\label{feleven}
\mu_\infty(D)\le\prod^n_{i=1} P_i(D)\end{equation}
but since $\lim_{j\to\infty} \mu_j(D)=\mu_\infty(D)$ and
$\mu_{j^*}(D)=\prod^n_{i=1} P_i(D)$, \eqref{feleven} follows by 
Theorem~\ref{thmtthree}(i).\end{proof}

\begin{coro}\label{corffour}  
If $P_1,\dots, P_n\in\calp$  are discrete with common atoms 
$A\ne\emptyset$, then $\&(P_1,\dots, P_n)$ is the
unique Borel probability distribution that minimizes the maximum loss
of Shannon Information between single Borel probability distributions
and $P_1, P_2,\dots,P_n$.
\end{coro}

\begin{proof}
 It is easy to check that for discrete distributions  $P_1,\dots, P_n$ with common
atoms $A$, $\|\mu_\infty(P_1,\dots, P_n)\|=\sum_{x\in A}\prod^n_{i=1} P_i(x)$,
which by the definition of $A$ is strictly positive. The conclusion
then follows immediately from Theorems~\ref{thmthone}  and \ref{thmfthree}.
\end{proof}

\begin{thm}\label{thmffive}
 If $P_1,P_2,\dots, P_n$ are a.c.\ with densities $f_1,\dots, f_n$,
satisfying $$0<\int^\infty_{-\infty}\prod^n_{i=1} f_i (x)dx< \infty,$$
 then there are Borel probability distributions $\{P_{i,j}:i\in \{1,\dots, n\},
j\in\bbn\}$ such that
\begin{itemize}
\item[\rm (i)] for all $i$, $P_{i,j}$ converges vaguely to $P_i$ as
$j\to\infty$, 
\item[\rm (ii)] $\&(P_{1,j},\dots, P_{n,j})$ is the unique minimizer of
loss of Shannon Information from $P_{1,j},\dots, P_{n,j}$, and 
\item[\rm (iii)] $\&(P_1,\dots, P_n)$ is the vague limit of 
$\&(P_{1,j},\dots, P_{n,j})$ as $j\to\infty$.
\end{itemize}
\end{thm}

\begin{proof}
For each $i\in \{1,\dots, n\}$ and $j\in \bbn$, let $P_{i,j} = 
\mu_j (P_i)$, and
note that $\mu_j(P_i)$ is a discrete p.m.\ for all $i$ and $j$, and
by Theorem~\ref{thmtthree}(iv), $\mu_j(P_i)\to P)_i$ vaguely as
$j\to\infty$ , which proves (i). Since $\{P_{i,j}:i\in \{1,\dots, n\}\}$ are compatible
for all $j\in\bbn$, $\mu_j(P_1),\dots, \mu_j(P_n)$ 
are discrete with at least one common atom, so by Theorem~\ref{thmthone} 
and Corollary~\ref{corffour}, $\& (P_{1,j},\dots, P_{n,j})=
\sum_{k\in Z}\prod^n_{i=1} P_i( (k-1)2^{-j}, k2^{-j}]$ 
is the unique minimizer of the maximum loss of Shannon
Information between single Borel p.m.'s and $\{P_{i,j}:i\in \{1,\dots, n\}\}$, 
which proves (ii). Finally,
note that for all $j\in\bbn$, 
$\prod^n_{i=1} P_i ((k-1)2^{-j},k2^{-j}]=\prod^n_{i=1}\mu_j(k2^{-j})$, 
so by the definition of $\{\mu_j\}$,
$\mu_j(P_1,\dots, P_n)=\sum_{k\in Z}\prod^n_{i=1}\mu_j(k2^{-j})\delta_{k2^{-j}}$,
and $\|\mu_j(P_1,\dots, P_n)\|=\sum_{k\in\bbz} \prod^n_{i=1} \mu_j (k2^{-j})>0$.
Hence, by Theorem~\ref{thmththree},
$$\&(P_{1,j},\dots, P_{n,j})=
\frac{\sum_{k\in\bbz}\prod^n_{i=1}\mu_j(k2^{-j})\delta_{k2^{-j}}
}{ \sum_{k\in\bbz}\prod^n_{i=1}\mu_j(k2^{-j})}=
\frac{\mu_j(P_1,\dots, P_n)}{ \|\mu_j(P_1,\dots, P_n)\|}$$
      converges vaguely to $\&(P_1,\dots, P_n)$, proving (iii).\end{proof}

\section{Minimax Likelihood Ratio Consolidations and Proportional Consolidations}

In classical hypotheses testing, a standard technique to decide
from which of $n$ known distributions given data actually came is to
maximize the likelihood ratios, that is, the ratios of the p.m.f.'s
or p.d.f.'s. Analogously, when the objective is not to decide from
which of $n$ known distributions $P_1,\dots, P_n$ the data came, but rather to decide how
best to consolidate data from those input distributions into a single
(output) distribution $P$, one natural criterion is to choose $P$ so as to
make the ratios of the likelihood of observing $x$ under  $P$ to the likelihood
of observing $x$ under {\it all} of the (independent) distributions $\{P_i\}$ as close as
possible. This motivates the notion of minimax likelihood ratio.

\begin{dfn}\label{dfnfvone} 
A discrete probability distribution $P^*\in\calp$  (with p.m.f.\ $p^*$)
is the {\it minimax likelihood ratio (MLR) consolidation of	discrete
distributions}  $P_1,\dots, P_n$ (with p.m.f.'s $\{p_i\}$) if
$$
\min_{\mbox{p.m.f.'s}\; p}\left\{
\max_{x\in\bbr}\frac{ p(x)}{\prod^n_{i=1}p_i(x)}-\min_{x\in\bbr}
\frac{p(x)}{
\prod^n_{i=1} p_i(x)}\right\}$$
is attained by	$p=p^*$ (where $0/0:=1$).  Similarly, an a.c.\ distribution 
$P^*\in\calp$  (with p.d.f.\ $f^*$)
is the {\it MLR consolidation of a.c.\ distributions} $P_1,\dots, P_n$ 
(with p.d.f.'s $f_1,\dots, f_n$) if
$$
\min_{\mbox{p.d.f.'s}\; f}\left\{
\esssup_{x\in\bbr} \frac{f(x)}{\prod^n_{i=1} f_i(x)}-
\essinf_{x\in\bbr} \frac{f(x)}{\prod^n_{i=1}f_i(x)}\right\}$$
is attained by $f^*$.
\end{dfn}

The min-max terms in \eqref{fvone}  and 
\eqref{fvtwo}  are similar to the min-max criterion
for loss of Shannon Information (Theorem~\ref{thmfthree}), 
whereas the others are
dual max-min criteria. Just as conflation was shown to minimize the loss
of Shannon Information, conflation will now be shown to also be the MLR
consolidation of the given input distributions.

\begin{thm}\label{thmfvtwo} 
If $P_1,\dots, P_n\in\calp$  are discrete with at least one common atom, or are
a.c.\ with p.d.f.'s $\{f_i\}$ satisfying $0<\int\prod^n_{i=1}f_i(x)dx<\infty$,
then $\&(P_1,\dots, P_n)$ is the unique MLR consolidation of $P_1,\dots, P_n$.
\end{thm}

\begin{proof}  
First consider the discrete case, let $\{p_i\}$ denote the p.m.f.'s of 
$\{P_i\}$,
 respectively, and let $\emptyset\ne A\subset\bbr$ 
denote the common atoms of $\{P_i\}$, i.e.\ $A=\left\{x\in\bbr:
\prod^n_{i=1}p_i(x)\right\}>0$.  
By Theorem~\ref{thmthone}, $\&(P_1,\dots, P_n)$ 
 is discrete with p.m.f.\ $p^*(x)=
\frac{\prod^n_{i=1}p_i(x)}{ \sum_{y\in A}\prod^n_{i=1}
p_i(y)}$. For each p.m.f.\ $p$, let
$$\udel (p)=\sup_{x\in\bbr} \frac{p(x)}{\prod^n_{i=1} p_i(x)} - 
\inf_{x\in\bbr} \frac{p(x)}{\prod^n_{i=1}p_i(x)}\,.$$
 Then, since $p^*(x)=0$ for every $x\in A^c$, it follows from the definition of 
$p^*$ (and the 
 convention $0/0:=1$) that $\udel (p^*)=
\left(\sum_{y\in A}\prod^n_{i=1} p_i(y)\right)^{-1}-1\ge 0$.  
Thus, to establish the theorem for
$P_1,\dots, P_n$ discrete,
 it suffices to show that for all p.m.f.'s $p$
\begin{equation}\label{fvone}
\udel (p)\ge\left(\sum_{y\in A}\prod^n_{i=1} p_i(y)\right)^{-1}-1,
\hbox{ with equality if and only if } p=p^*.\end{equation}
If  $\sum_{y\in A} p(y)<1$ then there exists an
$x_0\in A^c$ with $p(x_0)>0$, so $\frac{p(x_0)}{ \prod^n_{i=1} p_i(x_0)}=\infty$ 
and $\udel (p)=\infty$, so \eqref{fvone}  is trivial. On the other
hand if  $\sum_{y\in A} p(y)=1$, then $\min_{x\in \bbr}
\frac{ p(x)}{\prod^n_{i=1} p_i(x)}\le 1$
which implies that  $\udel (p)\ge\max_{x\in\bbr}\frac{ p(x)}{
\prod^n_{i=1} p_i(x)}-1$ for all $p$, and the argument in the proof
of Theorem~\ref{thmfthree} shows equality holds if and only
$\frac{p(x)}{\prod^n_{i=1} p_i(x)}$ is constant, i.e.\ if and
only if $p=p^*$. This proves \eqref{fvone} and completes the argument when 
$\{P_i\}$ are discrete.

For the a.c.\ conclusion, fix $\{P_i\}$ a.c.\ with p.d.f.'s satisfying 
\newline $0<\int\prod^n_{i=1}f_i(x)dx<\infty$. By 
Theorem~\ref{thmththree}  $\& (P_1,\dots, P_n)$ is a.c.\ with p.d.f.\
$f^*(x)=\frac{\prod^n_{i=1}f_i(x)}{\int\prod^n_{i=1} f_i(y)dy}$. For each p.d.f.\ 
$f$, let
$$
\udel (f)=\esssup_{x\in\bbr} \frac{f(x)}{ \prod^n_{i=1}f_i(x)}-
\essinf_{x\in\bbr}\frac{f(x)}{\prod^n_{i=1}f_i(x)}.$$
{\bf Case 1.} $\int\prod^n_{i=1}f_i(x)dx\in (0,1]$, $\prod^n_{i=1}f_i(x)>0$
a.s.\ (e.g., $\{P_i\}$ arbitrary normal distributions). Then since 
$\prod^n_{i=1}f_i(x)>0$, $\frac{f^*(x)}{\prod^n_{i=1}f_i(x)}=
\frac{1}{ \int\prod^n_{i=1} f_i(y)dy}$, a.s.,
which is constant, so $\udel (f^*)=0$. Thus it suffices to show that for all $f$  as in
Case~1,
\begin{equation}\label{fvtwo}
\udel f(x)\ge 0\quad\hbox{ with equality if and only if }
f=f^*.\end{equation} 
If $f$ is not positive a.s., then 
$\ess\inf \frac{f}{\prod^n_{i=1}f_i}=0$ since  
$\prod^n_{i=1}f_i(x)>0$ a.s., so $\udel (f)=\ess\sup 
\frac{f}{ \prod^n_{i=1}f_i}>0$, 
and the inequality in
\eqref{fvtwo}  is satisfied. On the other hand, if  $f>0$ a.s., then 
$\udel (f)=\ess\sup_{x\in\bbr} \frac{f(x)}{\prod^n_{i=1}f_i(x)}-
\ess\inf_{x\in\bbr} \frac{f(x)}{\prod^n_{i=1}f_i(x)}\ge 0$, with equality if
and only if  $\frac{f(x)}{\prod^n_{i=1}f_i(x)}$ is constant a.s.; i.e.\ if and only if  
$f=f^*$ a.s., which completes
the argument for Case~1.  

The three other cases 
\begin{align*}
&\left\{\int\prod^n_{i=1}f_i(x)dx\in (0,1], 
\prod^n_{i=1}f_i(x) \hbox{ not } >0 \hbox{ a.s.}\right\},\\ 
&\left\{ \int\prod^n_{i=1} f_i(x)dx\in (1,\infty), \prod^n_{i=1}f_i(x)
>0\hbox{ a.s.}\right\},\\
&\left\{ \int\prod^n_{i=1}f_i(x)dx\in (1,\infty),
\prod^n_{i=1}f_i(x)\hbox{ not } >0 \hbox{ a.s.}\right\}\end{align*} 
follow similarly.\end{proof}

If the $\{P_i\}$ are a.c.\ but do not satisfy the integrability condition in the
hypotheses of Theorem~\ref{thmfvtwo}, both parts of the conclusion of 
Theorem~\ref{thmfvtwo}
may fail: the conflation may not be MLR; and MLR distributions may not
be unique.

\begin{exam}\label{exfvthree}  
Let $n=2$, and $P_1=P_2$ be as in Example~\ref{exthsix}, 
so the conflation $\&(P_1,P_2)$ exists
and is $\delta_0$, which is not MLR for $P_1,P_2$ since it is not even a.c. However, every
a.c.\ distribution with p.d.f.\ $f_\alpha (x)=\alpha x^{\alpha-1}$ 
for  $x\in (0,1)$ (and $=0$ otherwise), $0<\alpha\le \frac14$, is MLR  for
$P_1,P_2$. To see
this, recall that $\prod^n_{i=1}f_i(x)=(4x)^{-1}$
for $x\in (0,1)$, and $=0$ otherwise. Thus  
$\frac{f_\alpha (x)}{\prod^n_{i=1} f_i(x)} = 4xf_\alpha(x)=4\alpha x^\alpha$ 
for $x\in (0,1)$, so $\ess\sup_{x\in \bbr}\frac{f_\alpha (x)}{
\prod^n_{i=1}f_i(x)}=1$,  since off $(0,1)$,  $\frac{f_\alpha (x)}{
\prod^n_{i=1} f_i(x)}=1$,
and on $(0,1)$,  $\ess\sup_{x\in\bbr}\frac{f_\alpha (x)}{\prod^n_{i=1}f_i(x)}=
4\alpha\le 1$.  Next, $\ess\inf_{x\in\bbr}\frac{f_\alpha (x)}{\prod^n_{i=1} f_i(x)}
=0$ since $\frac{f_\alpha (x)}{\prod^n_{i=1}f_i(x)}=4\alpha x^\alpha$ for 
$x\in (0,1)$.  Thus $\udel (f_\alpha)=1$, so to show $f_\alpha$ is MLR, requires showing
that $\udel (f)\ge 1$ for all p.d.f.'s $f$.  Fix $f$, and note that if 
$\ess\inf_{x\in\bbr}\frac{f(x)}{\prod^n_{i=1} f_i(x)}=\delta >0$, then on $(0,1)$, 
$\frac{f(x)}{\prod^n_{i=1} f_i(x)}=4xf(x)\ge \delta$
a.s., so $f(x)\ge \frac{\delta}{ 4x}$ 
a.s., which cannot be a density since it is not integrable.	Hence, 
$\ess\inf_{x\in\bbr}\frac{f(x)}{\prod^n_{i=1}f_i(x)}=0$.
But $\ess\sup_{x\in\bbr}\frac{f(x)}{\prod^n_{i=1}f_i(x)}\ge 1$,
 since $f$ is a.s.\ nonnegative and $\prod^n_{i=1}f_i(x)=0$
for all $x$ not in $(0,1)$. Thus $\udel (f)\ge 1$ so $f_\alpha$ is MLR.
\end{exam}

In the underlying problem of consolidating the independent distributions
$P_1,\dots, P_n$ into a single distribution $Q$, a criterion similar to MLR is to require
that $Q$ reflect the relative likelihoods of identical individual outcomes
under the $\{P_i\}$. For example, if the likelihood of all the experiments
$\{P_i\}$ observing the identical outcome $x$ is twice that of the likelihood of all
the experiments $\{P_i\}$ observing $y$, then $Q(x)$ should also be twice as large as
$Q(y)$. This motivates the notion of proportional consolidation.

\begin{dfn}\label{dfnfvfour}  
For discrete $P_1,\dots, P_n\in\calp$  with p.m.f.'s $p_1,\dots, p_n$, respectively, the
discrete distribution $Q \in\calp$  is a {\it proportional consolidation of}
$P_1,\dots, P_n$ if its
p.m.f.\ $q$ satisfies
$$\frac{q(x)}{ q(y)}=\frac{\prod^n_{i=1}p_i(x)}{\prod^n_{i=1}p_i(y)}\qquad
\hbox{for all }x,y\in\bbr.$$
Similarly, for a.c.\ $P_1,\dots, P_n\in\calp$  with p.d.f.'s 
$f_1,\dots, f_n$, respectively, the a.c.\ distribution
$Q \in\calp$ is a {\it proportional consolidation of} $P_1,\dots, P_n$ 
if its p.d.f.\ $g$ satisfies
$$\frac{g(x)}{ g(y)}=\frac{\prod^n_{i=1} f_i(x)}{\prod^n_{i=1}f_i(y}
\quad \hbox{for Lebesgue-almost-all }x,y\in\bbr.$$
\end{dfn}

\begin{thm}\label{thmfvfive} 
If $P_1,\dots, P_n\in \calp$  are discrete with at least one common atom, or are
a.c.\ with p.d.f.'s $\{f_i\}$ satisfying $0<\int\prod^n_{i=1} f_i(x)dx<\infty$, 
then the conflation $\&(P_1,\dots, P_n)$ is the unique
proportional consolidation of $P_1,\dots, P_n$.
\end{thm}

\begin{proof}  First consider the case where $\{P_i\}$ are discrete, and let $\{p_i\}$ be the
p.m.f.'s for $\{P_i\}$, respectively. By Theorem~\ref{thmthone} again, 
$\&(P_1,\dots, P_n)$ is discrete with p.m.f.\
$p^*(x)=\frac{\prod^n_{i=1}p_i(x)}{\sum_{y\in\bbr}\prod^n_{i=1}p_i(y)}$ 
for all $x\in\bbr$. Thus
$\frac{p^*(x)}{ p^*(y)} = \frac{\prod^n_{i=1} p_i(x)}{\prod^n_{i=1}
p_i(y)}$, so  $\&(P_1,\dots, P_n)$ is a proportional consolidation of 
$P_1,\dots, P_n$. To see that $\&(P_1,\dots, P_n)$ is
the unique proportional consolidation, suppose $Q\ne \&(P_1,\dots, P_n)$, 
and set $q(x)=Q(x)$ for all $x\in\bbr$. Since,
$Q\ne\&(P_1,\dots, P_n)$, it follows from Theorem~\ref{thmthone} that there exist 
$x,y\in\bbr$ so that $q(x)>\frac{\prod^n_{i=1}p_i(x)}{\sum_{z\in\bbr}
\prod^n_{i=1} p_i(z)}$ and
 $q(y)<\frac{\prod^n_{i=1}p_i(y)}{\sum_{z\in\bbr}\prod_{i=1}^np_i(z)}$, so 
$\frac{q(x)}{ q(y)}>\frac{\prod^n_{i=1}p_i(x)}{\prod^n_{i=1}p_i(y)}$, and 
$Q$ is not a proportional consolidation of $P_1,\dots, P_n$.  
The case where $P_1,\dots, P_n$ are a.c.\ follows
similarly, again using 
Theorem~\ref{thmththree}  in place of Theorem~\ref{thmthone}.\end{proof}

Here, too, the conclusion for a.c.\ distributions may fail if the
integrability hypothesis condition is not satisfied.

\begin{exam}\label{exfvsix}  
 Let $n=2$, and $P_1=P_2$ be as in Example~\ref{exthfive}, so again
$\prod^n_{i=1}f_i(x)=(4x)^{-1}$ for $x\in (0,1)$, and $=0$
otherwise. This implies that $\frac{\prod^n_{i=1}f_i(x)}{
\prod^n_{i=1}f_i(y)}=\frac{y}{ x}$ for Lebesgue almost all $x,y\in(0,1)$. 
But there are no
p.d.f.'s $f$ with support on $(0,1)$ such that $\frac{f(x)}{ f(y)}=\frac{y}{ x}$ 
a.s., since then for fixed
$y$,  $f(x)=\frac{yf(y)}{ x}$ for almost all $x\in(0,1)$, and 
$\int^1_0 cx^{-1}dx=0$ if $c=0$ and $=\infty$ if $c>0$. Thus, there is no proportional
consolidation of this $P_1,P_2$ (in contrast to the conclusion of 
Example~\ref{exfvthree}
for these same distributions, where it was seen that there are many
MLR consolidations).
\end{exam}

\section{Conflations of Normal Distributions}

In describing the method used to obtain values for the fundamental
physical constants from the input data, NIST explains that certain
data ``are the means of tens of individual values, with each value
being the average of about ten data points" \cite[p.~679]{MTN2}, and
predicates interpretation of some of their conclusions on the
condition ``If the probability distribution associated with each
input datum is assumed to be normal" \cite[p.~483]{MT1}.  After comparing
the most recent (2006) results from electrical watt-balance and from
silicon-lattice sphere experiments used to estimate Planck's constant,
however, NIST determined that the means and standard deviations of
several distributions of input data were not sufficiently close, and
reported that their ``data analysis uncovered two major inconsistencies
with the input data," conceding that the resulting official NIST 2006
set of recommended values for the fundamental physical constants ``does
not rest on as solid a foundation as one might wish" \cite[p.~54]{MTN1}.
In order to eliminate this perceived inconsistency, the NIST task
group ``ultimately decided that $\dots$ the a priori assigned uncertainties of
the input data involved in the two discrepancies would be weighted
by the multiplicative factor 1.5," which ``reduced the discrepancies
to a level comfortably between two standard deviations" \cite[p.~54]{MTN1}.

But if the various input distributions are all normal, for example,
as in the NIST assumption, then every interval centered at the
unknown positive true value of Planck's constant has a positive
probability of occurring in every one of the independent
experiments. If the input data distributions happen to have
different means and variances, that does not imply the input is
``inconsistent." Thus in consolidating data from several independent
sources, special attention should be paid to the normal case.

The conflation of normal distributions has several important
properties -- it is itself normal (hence unimodal), and in addition
to minimizing the loss of Shannon Information (Theorem~\ref{thmfthree}) and
being the unique MLR consolidation (Theorem~\ref{thmfvtwo}) and the unique
proportional consolidation (Theorem~\ref{thmfvfive}), the conflation of normal
distributions also yields the classical weighted mean squares and
best linear unbiased estimators for general unbiased data, and
maximum likelihood estimators for normally-distributed unbiased
input data.

\begin{thm}\label{thmsone}  
If $P_i$ is $N(\mu_i,\sigma^2_i)$, $i=1,\dots, n$, then
$$\&(P_1,\dots, P_n)=
N\left(\frac{\sum^n_{i=1}\frac{\mu_i}{\sigma^{2}_i}}{
\sum^n_{i=1}\frac{1}{\sigma^{2}_i}}\,
,\frac{1}{\sum^n_{i=1}\sigma^{-2}_i}\right).$$
\end{thm}

\begin{proof} By Theorem~\ref{thmththree},  
$\&(P_1,\dots, P_n)$ is a.c.\ with density proportional to the
product of the densities for each distribution, and the conclusion then
follows immediately from the definition of normal densities and a routine
calculation by completing the square.\end{proof}

\begin{exam}\label{exstwo}  
If $P_1$  is $N(1,1)$ and $P_2$ is $N(2,4)$, then $\&(P_1,P_2)$ is
$N(\frac65, \frac45)$.
\end{exam}

The mean of the conflations of normals given in 
Theorem~\ref{thmsone},\newline
 $\sum^n_{i=1}\mu_i\sigma^{-2}_i\left(\sum^n_{i=1}
\sigma^{-2}_i\right)$,
is precisely the value of the weighted least 
squares estimate 
given by Aitken's generalization of the Gauss-Markov Theorem, and this 
simple observation will next be exploited to obtain several conclusions 
relating conflation and statistical estimators. 

First, however, it must be remarked that the mean of the conflation is not 
in general the same as the weighted least squares estimate. 
Conflation disregards outlier or ``inconsistent" data values, whereas 
weighted least squares gives full weight to all values. For instance, 
if one of the input distributions includes negative entries 
(e.g., is reported as a true Gaussian), and the others do not, then 
conflation eliminates the negative values. The following
example for the uniform distribution illustrates this, and the same 
argument can easily be applied to other distributions such as 
truncated normals (Theorem~\ref{thmsvtwo} below). 

\begin{exam}\label{exsixthree}  
Let $P_1$ be $U(0,1)$ and $P_2$ be $U(-0.1,1)$. 
By Theorem~\ref{thmththree}, the conflation of $P_1$ and $P_2$ is 
$\&(P_1,P_2)=U(0,1)$, which ignores the negative values of 
$P_2$ and has mean $\frac12$. 
The weighted least squares estimate, however, is easily seen to be  
$\left(\frac{12}{ 1} + \frac{12}{ 1.1^2}\right)^{-1}\left(
\frac{12}{ 2} + \left(\frac{9}{20}\right)\left(
\frac{12}{ 1.1^2}\right)\right)<.48$.
\end{exam}

To establish the link between conflation and statistical estimators, 
recall that a random variable $X$ is an unbiased estimator of an unknown
parameter $\theta$ if $EX=\theta$, and note that if $X$ is a r.v., then  
$N(X,\sigma^2)$ is a {\it random normal
distribution} with variable mean $X$ and fixed variance $\sigma^2$.  

\begin{thm}\label{thmsthree}  
If $X_1,\dots, X_n$ 
are independent unbiased estimators of  $\theta$ with finite variances 
$\sigma^2_1,\dots, \sigma^2_n$,
respectively, then $\Theta=\hbox{mean}(\&(N_1,\dots, N_n))$ is 
the best linear unbiased estimator for 
$\theta$, where $\{N_i\}$ are the random normal distributions $N_i=N(X_i,\sigma^2_i)$,
$i=1,\dots, n$.
\end{thm}

\begin{proof}  By Theorem~\ref{thmsone},  $\&(N_1,\dots, N_n)$ is
\newline
$N\left(\sum^n_{i=1}\mu_i\sigma^{-2}_i\left(\sum^n_{i=1}\sigma^{-2}_i\right)^{-1},
\left(\sum^n_{i=1}\sigma^{-2}_i\right)^{-1}\right)$,
 where $\{\mu_i\}$  and $\{\sigma^2_i\}$ are the means and variances of 
$\{N_i\}$, respectively. Since $N_i$ is $N(X_i,\sigma^2_i)$ for each $i=1,\dots, n$,
where the  $\{X_i\}$ are r.v.'s, this implies that
$\&(N_1,\dots, N_n)$ is the random distribution\newline
$N\left(\sum^n_{i=1}X_i\sigma^{-2}_i\left(\sum^n_{i=1}\sigma^{-2}_i\right)^{-1},
\left(\sum^n_{i=1}\sigma^{-2}_i\right)^{-1}\right)$, so
\begin{equation}\label{sone}
\hbox{mean} (\&(N_1,\dots, N_n)=\left(\sum^n_{i=1}\sigma^{-2}_i\right)^{-1}
\sum^n_{i=1}X_i\sigma^{-2}_i.\end{equation}
Since the right hand side of \eqref{sone}  is the classical weighted least
squares estimator for $\theta$, Aitken's generalization of the Gauss-Markov
Theorem (e.g. \cite{Aitc}, \cite[Theorem~7.8a]{RS}) implies that it is the best linear unbiased estimator for 
$\theta$.\end{proof}

Note that normality of the distributions is in the {\it conclusion}, not
the hypotheses, of Theorem~\ref{thmsthree}. If, in addition, the underlying data
distributions are normal, this estimator is even a maximum likelihood estimator.  

\begin{thm}\label{thmsfour}  
If $X_1,\dots, X_n$ 
are independent normally-distributed unbiased estimators of  $\theta$ with finite
variances $\sigma^2_1,\dots, \sigma^2_n$, 
respectively, then $\Theta=\hbox{mean}(\&(N_1,\dots, N_n))$ is a maximum likelihood estimator for 
$\theta$, where $\{N_i\}$ are the random normal distributions
$N_i=N(X_i,\sigma^2_i)$, $i=1,\dots, n$.
\end{thm}

\begin{proof} Analogous
to proof of Theorem~\ref{thmsthree}, using \cite[Theorem~7.8b]{RS}.
\end{proof}

\section{Closure and Truncation Properties of Conflation}

If input data distributions are of a particular form, it is often
desirable that consolidation of the input also have that same
form. Theorem 6.1 showed that the conflation of normal distributions
is always normal, and the next theorem shows that many other
classical families of distributions are closed under conflation.

Recall that: a discrete probability distribution is {\it Bernoulli} with
parameter $p\in [0,1]$ if its p.m.f.\ is $p(1)=1-p(0)=p$, is {\it geometric}
 with parameter $p\in [0,1]$ if its
p.m.f.\ is $p(k)=(1-p)^{k-1}p$ for all $k\in\bbn$, is 
{\it discrete uniform} on $\{1,2,\dots, n\}$  if its p.m.f.\ is $p(k)=n^{-1}$
for all $k\in \{1,2,\dots, n\}$, is {\it Zipf}
with parameters $\alpha>0$ and  $n\in \bbn$ if its p.m.f.\ is proportional to 
$k^{-\alpha}$ for all $k\in \{1,2,\dots, n\}$, and is
{\it Zeta} with parameter  $\alpha>1$ if its p.m.f.\ is proportional to 
$k^{-\alpha}$ for all $k\in\bbn$; and an
a.c.\ probability distribution is {\it gamma} with parameters 
$\alpha\in\bbn$ and $\beta >0$ if its p.d.f.\ is
proportional to $x^{\alpha-1}e^{-x/\beta}$ for $x>0$, is {\it beta}
 with parameters $\alpha>1$ and $\beta>1$ if its p.d.f.\ is proportional
to $x^{\alpha-1}(1-x)^{\beta-1}$ for $0<x<1$, 
is {\it uniform} on $(a,b)$ for $a<b$ if its p.d.f.\ is constant 
$(b-a)^{-1}$ for $a<x<b$, is standard {\it LaPlace} (or
double-exponential) with parameter  $\alpha>0$ if its p.d.f.\ is proportional 
to
$e^{-|x|/\beta}$, $-\infty<x<\infty$,
is {\it Pareto} with parameters $\alpha>0$ and $\beta >0$ if its p.d.f.\
 is proportional to $x^{-(\alpha+1)}$ for $\beta < x<\infty$, and is
{\it exponential} with mean $a>0$ if its p.d.f.\ is proportional to
$e^{-x/\alpha}$ for $x>0$.

\begin{thm}\label{thmsvone}  
Let $P_1,P_2,\dots, P_n$ be compatible.
\begin{itemize}
\item[\rm (i)] If $\{P_i\}$ are {\rm Bernoulli} with parameters 
$\{p_i\}$ respectively, then
\item[]
$\&(P_1,\dots,P_n)$ 
is Bernoulli with parameter 
$p=\frac{\prod^n_{i=1}p_i}{\left(\prod^n_{i=1}p_i+\prod^n_{i=1}(1-p_i)\right)}$.  
\item[\rm (ii)] If $\{P_i\}$ are {\rm geometric} with parameters
$\{p_i\}$ respectively,\item[] then
$\&(P_1,\dots, P_n)$ is geometric with parameter
$p=1-\prod^n_{i=1}(1-p_i)$.
\item[\rm (iii)] If $\{P_i\}$ are {\rm discrete uniform} on $\{1,\dots, n_i\}$
respectively,\item[] then
$\&(P_1,\dots, P_n)$ is uniform on $\{1,\dots, \min_i\{n_i\}\}$.
\item[\rm (iv)] If $\{P_i\}$ are {\rm Zipf} with parameters 
$\{\alpha_i\}$ and $\{n_i\}$, respectively, then 
$\&(P_1,\dots, P_n)$ is Zipf with parameters  $\alpha=\sum^n_{i=1}\alpha_i$
and $n=\min_i\{n_i\}$.
\item[\rm (v)]If $\{P_i\}$ are {\rm Zeta} with parameters 
$\{\alpha_i\}$ respectively, then 
$\&(P_1,\dots, P_n)$ is Zeta with parameter  $\alpha=\sum^n_{i=1}\alpha_i$.
\item[\rm (vi)] If $\{P_i\}$ are {\rm gamma} with parameters 
$\{\alpha_i,\beta_i\}$
respectively,\item[] then
$\&(P_1,\dots, P_n)$ is gamma with parameters $\alpha=\sum^n_{i=1}\alpha_i-
(n-1)$, $\beta=\left(\sum^n_{i=1}(\beta_i)^{-1}\right)^{-1}$.
\item[\rm(vii)]If $\{P_i\}$ are {\rm beta} with parameters $\{\alpha_i,\beta_i\}$
respectively, then
$\&(P_1,\dots, P_n)$ is beta with parameters $\alpha=\sum^n_{i=1}\alpha_i-(n-1)$,
$\beta=\sum^n_{i=1}\beta_i-(n-1)$.
\item[\rm(viii)]If $\{P_i\}$ are {\rm continuous uniform} on intervals 
$\{(a_i,b_i)\}$
respectively, then
$\&(P_1,\dots, P_n)$ is uniform on $(\max_i a_i,\min_ib_i)$.
\item[\rm(ix)]If $\{P_i\}$ are {\rm LaPlace} with parameters $\{\alpha_i\}$ respectively, then
\item[]
$\&(P_1,\dots, P_n)$ is LaPlace with parameter $\alpha=\left(\sum^n_{i=1}(\alpha_i)^{-1}
\right)^{-1}$.
\item[\rm(x)]If $\{P_i\}$ are {\rm Pareto} with parameters $\{\alpha_i,\beta_i\}$
respectively, then 
\item[]
$\&(P_1,\dots, P_n)$ is Pareto with parameters $\alpha=\sum^n_{i=1}\alpha_i+n-1$
and $\beta=\max_i\beta_i$.  
\item[\rm(xi)]If $\{P_i\}$ are {\rm exponential} with means $\{\alpha_i\}$ 
respectively, then 
$\&(P_1,\dots, P_n)$ is exponential with mean 
$\alpha=\left(\sum^n_{i=1}\alpha^{-1}_i\right)^{-1}$.
\end{itemize}\end{thm}

\begin{proof} Conclusions (i)--(v) follow from 
Theorem~\ref{thmthone}  and routine
calculations, and (vi)--(xi) follow from 
Theorem~\ref{thmththree}  and calculations.\end{proof}

Note that for smaller values of the parameters of beta distributions, the
conflation may not be beta simply because the product of the densities
may not be integrable. The families of distributions identified in
Theorem~\ref{thmsvone}
  that are closed under conflation are by no means exhaustive.
For example, the conflation of $n$ Poisson distributions is not classical
Poisson, but is a discrete Conway-Maxwell-Poisson (CMP) distribution
with p.m.f.\ proportional to $\frac{\lam^k}{(k!)^n}$, $k=0,1,\dots$ 
and clearly the CMP family is closed
under conflation.

Recall that the conflation of Cauchy distributions is not Cauchy,
as was shown in Example~\ref{exthfive}. It is easy to see that the families of
binomial distributions and of chi-square distributions are not closed
under conflation, but chi-square comes very close in the following sense:
if $X$ is a random variable with distribution $\&(P_1,\dots, P_n)$ 
where $\{P_i\}$ are chi-square with $\{k_i\}$ degrees
of freedom, respectively, then $X/n$ is chi-square with 
$\sum^n_{i=1}k_i-2n+2$ degrees of freedom.

In practice, assumptions are often made about the form of the input
distributions, such as NIST's essential assumption that underlying
data is often normally distributed. But the true and estimated values
for Planck's constant clearly are never negative, so the underlying
distribution is certainly not truly normally distributed -- more
likely it is truncated normal. The additional assumption of exact
normality, in addition to their use of linearizing the observational
equations and then applying generalized least squares \cite[p.~481]{MT1},
introduces further errors into the NIST estimates.

Using conflations, however, the problem of truncation essentially
disappears -- it is automatically taken into account. The reason is
that another important feature of conflations is that it preserves
many classes of truncated distributions, where a distribution
of a certain type is called {\it truncated} if it is the conditional
distribution of that type conditioned to be in a (finite or
infinite) interval. For example, truncated normal distributions
include normal distributions conditioned to be positive (that is,
a.c.\ distributions with density function proportional to 
$e^{-(x-\mu)^2/2\sigma^2}$, $x>0$ (and zero elsewhere)), as is
often the case in experimental data involving estimates of many of
the fundamental physical constants.

\begin{thm}\label{thmsvtwo}  
If $P_1,P_2,\dots, P_n$ are compatible truncated normal (exponential, gamma,
LaPlace, Pareto) distributions, then $\&(P_1,P_2,\dots, P_n)$ is also a truncated normal
(exponential, gamma, LaPlace, Pareto, respectively) distribution.
\end{thm}

\begin{proof} Immediate from Theorem~\ref{thmththree}.\end{proof}

The above example of determination of the values of the fundamental
physical constants is only one among many scientific situations
where consolidation of dissimilar data is problematic. Some
government agencies, such as the Methods and Data Comparability
Board of the National Water Quality Monitoring Council \cite{MDC}, have
even established special programs to address this issue. Perhaps
the method of {\it conflating} input data will provide a practical and
simple, yet optimal and rigorous method to address this problem.

\section*{Acknowledgement}

            The author wishes to express his gratitude to Professors
            Matt Carlton, Dan Fox, Jeff Geronimo, Kent Morrison, Carl
            Spruill, and John Walker for valuable advice and suggestions,
            to Editor Steve Spicer of Water, Environment and Technology,
            and especially to Professor Ron Fox for his extensive input.
The author is
            also grateful to the Free University of Amsterdam and the
            California Polytechnic State University in San Luis Obispo,
            where much of this research was accomplished.


\begin{thebibliography}{999}


\bibitem{Aitc} Aitchison, J. (1982) The statistical analysis of compositional
data, {\it J. Royal. Statist. Soc. Series B \bf 44}, 139--177.

\bibitem{Aitk} Aitken, A. (1934) On least-squares and linear combinations of
observations, {\it Proc. Royal Soc. Edinburgh \bf 55}, 42--48.

\bibitem{B} Bracewell, R. (1999) {\it The Fourier Transform and its Applications},
3rd Ed., McGraw-Hill.

\bibitem{C} Chung, K.L. (2001) {\it A Course in Probability Theory}, 3rd Ed. Academic
Press, New York.

\bibitem{EDP} Egozcue, J., Diaz-Barrero, J. and Pawlowsky-Glahn, V. (2006)
Hilbert space of p.d.f.s based on Aitchison geometry, {\it Acta Math. Sinica
\bf 22}, 1175--1182.

\bibitem{EH} Elton, J. and Hill, T., (1987) A generalization of Lyapounov's
convexity theorem to measures with atoms, {\it Proc. Amer. Math. Soc. \bf 99},
297--304.  

\bibitem{GZ} Genest,C. and Zidek, J. (1986) 
Combining probability distributions: a critique and an annotated 
bibliography, {\it Statist. Sci. \bf 1}, 114--148.

\bibitem{Lu} Lubinski, P. (2004) Averaging spectral shapes, {\it Monthly
Notices of the Royal Astronomical Society
\bf 350}, 596--608.  

\bibitem{Ly} Lyapounov, A. (1940) Sur les fonctions-vecteurs
completement additives, {\it Bull. Acad. Sci URSS Math \bf 4}, 465--478.

\bibitem{MDC} {\tt http://acwi.gov/methods/about/background.html}

\bibitem{MT1} Mohr, P. and Taylor, B. (2000) 
CODATA recommended values of the fundamental
physical constants:1998,  {\it Rev. Mod. Physics \bf 72}, 351--495.  

\bibitem{MTN1} Mohr, P., 
Taylor, B. and Newell, D. (2007) The fundamental physical constants,
{\it Physics Today},   52--55.  

\bibitem{MTN2} Mohr, P.,Taylor, B. and Newell, D. (2008) 
Recommended values of the fundamental physical constants:2006, 
{\it Rev. Mod. Phys. \bf 80}, 633--730.

\bibitem{RS} Rencher, A. and Schaalje, G. (2008) {\it Linear
Models in Statistics}, Wiley.  

\end{thebibliography}
\end{document}